\documentclass[11pt]{article}
\usepackage{}
\usepackage[round]{natbib}
\usepackage{mathrsfs}
\usepackage{caption}
\usepackage{amsfonts}
\usepackage{tabularray}
\usepackage{amsmath}
\usepackage{amssymb}
\usepackage{dsfont,footmisc}
\usepackage[all]{xy}
\usepackage{graphicx}
\usepackage{CJK}
\usepackage{color}
\usepackage{subcaption}
\usepackage[singlespacing]{setspace}
\usepackage{multirow}
\usepackage{rotating}
\usepackage{comment}
\usepackage{tabularx}
\usepackage[font=small]{caption}%
\usepackage[colorlinks,linkcolor=red,anchorcolor=blue,citecolor=blue]%
{hyperref}
\def\d{\mathrm{d}}
\def\I{\mathbf{I}}
\newcommand{\R}{\mathbb{R}}
\newcommand{\N}{\mathbb{N}}
\newcommand{\p}{\mathbb{P}}
\newcommand{\E}{\mathbb{E}}
\renewcommand{\(}{\left(}
\renewcommand{\)}{\right)}
\renewcommand{\[}{\left[}
\renewcommand{\]}{\right]}
\allowdisplaybreaks

\newtheorem{theorem}{Theorem}[section]

\newtheorem{corollary}{Corollary}[section]

\newtheorem{definition}{Definition}[section]

\newtheorem{lemma}{Lemma}[section]

\newtheorem{remark}{Remark}[section]

\newcommand{\VaR}{\mathrm{VaR}}
\newcommand{\RV}{\mathcal{RV}}
\newenvironment{proof}[1][Proof]{\noindent \textbf{#1.} }{\  \rule{0.5em}{0.5em}}
\newtheorem{proposition}{Proposition}[section]
\def\nn{\nonumber}

\title{Second-order Asymptotic Analysis of Tail Probabilities of Bidimensional Randomly Weighted Sums
}
\author{Bingzhen Geng\thanks{\scriptsize  School of Big Data and Statistics, Anhui University, Hefei, Anhui, 230601, China. Email: \texttt{gengbz@ahu.edu.cn}}
	\and
	Yang Liu\thanks{\scriptsize  
    School of Science and Engineering, The Chinese University of Hong Kong (Shenzhen), Shenzhen, Guangdong, 518172, China. 
    Email: \texttt{yangliu16@cuhk.edu.cn}}
 \and
    Shijie Wang\thanks{\scriptsize Corresponding Author. School of Big Data and Statistics, Anhui University, Hefei, Anhui, 230601, China. Email: \texttt{wangshijie@ahu.edu.cn}}
}

\date{}
\begin{document}

\maketitle
\begin{abstract}
Motivated by a bidimensional discrete-time risk model in insurance, we study second-order asymptotics for two kinds of tail probabilities of the stochastic discounted value of aggregate net losses including two business lines. These are essentially modeled as randomly weighted sums $S_n^{\xi}=\sum_{i=1}^n\xi_iX_i$ and $T_m^\eta=\sum_{j=1}^m\eta_{j}Y_{j}$ for any fixed $n,m\in\N$, in which it is assumed that the primary random variables $\left\{\(X,Y\),\(X_i,Y_i\):i\in \N\right\}$ form a sequence of real-valued, independent and identically distributed random pairs following a common bivariate Farlie-Gumbel-Morgenstern distribution and the random weights $\left\{\xi_i,\eta_i:i\in \N\right\}$ are bounded, nonnegative and arbitrarily dependent, but independent of the primary random variables. Under the assumption that two marginal distributions of the primary random variables are second-order subexponential, we first obtain the second-order asymptotic formulas for the joint and sum tail probabilities, which generalize and strengthen some known ones in the literature. Furthermore, by directly applying the obtained results to the above bidimensional risk model, we establish second-order asymptotic formulas for the corresponding tail probabilities. Compared with the first-order ones, our numerical simulation shows that the second-order asymptotics are much more precise. 

\quad 

\end{abstract}

\section{Introduction}
The tail behavior of a loss variable or distribution, often interpreted as the ruin probability, 
is one of the central topics in applied probability, especially in extreme value theory. In the past two decades, the one-dimensional discrete-time financial and insurance risk model has been investigated by many papers; see \cite{tang2003precise,tang2004ruin,su2006behavior} and so on. However, operating with a single line of business may not be a practical model for large insurers, which drives the development of models on surplus processes for multiple lines; see, e.g., \cite{geng2025asymptotics}. 

In the field of extreme value theory, the randomly weighted sum is an important research topic, and the asymptotic tail behavior of a single randomly weighted sum (later $S_n^\xi$) has been extensively studied in the presence of primary heavy-tailed random variables (rvs); see \cite{tang2003randomly,zhang2009approximation,tang2014randomly,cheng2018randomly,geng2019tail} and so on. 
Our paper contributes to the second-order asymptotic theory of the joint tail behavior of two randomly weighted sums in a general model:
\begin{align}\label{eq:sum}
S_n^{\xi}=\sum_{i=1}^n\xi_iX_i~~~~\mbox{and}~~~~T_m^{\eta}=\sum_{j=1}^m\eta_jY_j
\end{align}
for any fixed $n,m\in\N$. Here, 
the random variable $X_i$ is generally interpreted as the net loss in period $i$, and the random weight $\xi_i$ is called the accumulated discount factor up to time $i$. Then $S_n^\xi$ describes the present value of the total risk in the first $n$ periods. Parallel to the first random tuple $(X_i, \xi_i, S_n^\xi)$, we denote the second random tuple by $(Y_j, \eta_j, T_m^\eta)$, which can be interpreted as the second business line of a large insurer later. Specifically, we are interested in the second-order asymptotic expansions for the tail probabilities of 
$$
\p\(S_n^{\xi}>x,T_m^{\eta}>y\)~~~~\mbox{and}~~~~ \p\(S_n^{\xi}+T_m^{\eta}>x+y\)
$$ 
as $(x,y)\rightarrow(\infty,\infty)$, which are called the joint tail probability and the sum tail probability in this paper, respectively.

We first highlight an important justification for studying the joint tail probability of the pair
$(S_n^{\xi}, T_m^{\eta})$ in generality, with possibly different time horizons $n$ and $m$.
In the study of systemic risk, a central quantity of interest is the conditional tail probability
$$
\p\!\left(S_n^{\xi} > x \,\middle|\, T_m^{\eta} > y\right)
  = \frac{\p\!\left(S_n^{\xi} > x,\, T_m^{\eta} > y\right)}{\p\!\left(T_m^{\eta} > y\right)},
$$
which measures the likelihood that one financial entity (or business line) suffers an extreme loss, given that another entity has already experienced a tail event. Such conditional probabilities have natural interpretations in
systemic risk: for instance, $\p(S_n^{\xi} > x \mid T_m^{\eta} > y)$ quantifies the contagion risk or spillover probability from one institution to another, and serves as a building block for systemic risk measures such as CoVaR and MES (Marginal Expected Shortfall); see, e.g., 
\cite{man2024tail} and \cite{geng2025asymptotics,geng2024value}. 
Note that $n$ and $m$ in this conditional probability can naturally differ,
as the two entities may operate over different time horizons or reporting periods.
The key to computing (and to obtaining second-order asymptotics for) such a
conditional probability is precisely the joint tail probability
$\p(S_n^{\xi} > x,\, T_m^{\eta} > y)$ studied in the present paper. Therefore, the second-order asymptotic results for~\eqref{eq:sum} established
here have direct implications for the precision of conditional tail probability estimates in systemic risk measurement.

Next, as a motivating application of Eq. \eqref{eq:sum}, we introduce a concrete and widely-recognized bidimensional discrete-time risk model \eqref{eq:U} in insurance. 
Within the period $i = 1, \cdots, n$, we denote by the real-valued primary random vectors $\(X_i,Y_i\)$ the net losses of the insurer (the total amount of the claims minus the total premiums) from two different business lines; the positive random weight vectors $(R_i,\widetilde{R}_i)$  denote the stochastic discount factors during period $i-1$ to period $i$. The random vectors $\{(X_i,Y_i):i\in\N\}$ represent insurance
risks, and $\{(R_i,\widetilde{R}_i):i\in\N\}$ represent financial risks (\cite{nyrhinen1999ruin,nyrhinen2001finite}). In this setup, the bidimensional stochastic discounted value of the aggregate net losses up to
time $n$ can be characterized by 
\begin{align}\label{eq:U}
\left\{
\begin{aligned}
D_1(n)&=\sum_{i=1}^nX_i\prod_{k=1}^iR_k,\\
D_2(n)&=\sum_{i=1}^nY_i\prod_{k=1}^i\widetilde{R}_k.
\end{aligned}
\right.
\end{align}
We are
 concerned with the joint tail probability and the sum tail probability of such a bidimensional discrete-time risk model, defined, respectively, as
\begin{align}\label{eq:D}
\left\{
\begin{aligned}
D_{\text{and}}(x,y;n)&=\p\(D_1(n)>x, D_2(n)>y\),\\
D_{\text{sum}}(x,y;n)&=\p\(D_1(n)+D_2(n)>x+y\),
\end{aligned}
\right.
\end{align}
where $x>0$ and $y>0$ are the initial surpluses of the two business lines, respectively. 
If we denote the
aggregate stochastic discount factors by
$\xi_i=\prod_{k=1}^iR_k$ 
and $\eta_i =\prod_{k=1}^i\widetilde{R}_k$ 
during the first $i$ periods and $m=n$, then Eq. \eqref{eq:U} becomes a special case of Eq. \eqref{eq:sum}.

We note that in the above definition, 
$D_{\text{and}}(x,y;n)$ represents the probability that the aggregate discounted net losses exceed the initial surpluses in two business lines at the same time $n$, while $D_{\text{sum}}(x,y;n)$ represents the probability that the total aggregate discounted net losses exceed the total initial surplus up to time $n$. Especially, if we consider nonnegative random vectors $\{\(X_i,Y_i\):i\in\N\}$, then $D_{\text{and}}(x,y;n)$ reflects that the ruin occurs in both business lines over time $n$ and $D_{\text{sum}}(x,y;n)$ shows the probability of the total ruin.  
For detailed discussions, we refer readers to \cite{qu2013approximations,sun2014finite,tang2016random,chen2017interplay} and so on.

Indeed, in today's insurance industry, an insurance company often operates multiple different business lines simultaneously. The insights gained from studying these models in a two-variable scenario can be valuable in understanding and addressing these challenges. 
Therefore, \cite{li2018joint} considered the joint tail of \eqref{eq:sum}, in which each pair of $(X_i,Y_i)$ is \emph{strongly asymptotically independent} and dominatedly-varying-tailed. They obtained first-order asymptotic expansions for the joint tail of \eqref{eq:sum} with fixed positive integers $n,m\in\N$. Furthermore, when the marginal distributions of $(X_i,Y_i)$ are restricted to the regularly-varying class (see Definition \ref{de:3}), it was shown that the corresponding first-order asymptotic formula holds uniformly for any $n,m\in\N\cup\{\infty\}$.  More recently, \cite{yang2024asymptotics} further studied the first-order {\color{blue}asymptotics} for the joint tail of \eqref{eq:sum} for some fixed  $n,m\in\N$ by allowing pairs of $(X_i,Y_i)$ to be \emph{pairwise quasi-asymptotic independent} and the marginal distributions of each pair of $(X_i,Y_i)$ belonging to the intersection of long-tailed and dominatedly-varying-tailed class. Then, the uniformly asymptotic results were also investigated in the presence of consistently-varying-tailed primary rvs.
 
In alignment with the bivariate structure of Eq. \eqref{eq:sum}, we highlight its significance through two additional application scenarios in quantitative risk management.  
First, in the field of actuarial science, \eqref{eq:sum} is particularly relevant to Asset-Liability Management (ALM). An insurer may invest in multiple risky assets. If $S_n^{\xi}$ represents the discounted value of future insurance net losses (liabilities) and $T_m^{\eta}$ represents the fluctuating value of $m$ investment assets, the two-dimensional framework allows for the assessment of the joint tail risk. In terms of $T_m^{\eta}$, in static portfolio optimization, a portfolio contains $m$ risky assets within a single period (see, e.g., \cite{chen2022ordering,blanchet2025convolution}). For $j = 1, \dots, m$, the $j$-th asset suffers a loss $Y_j$ at the end of the period and is discounted by a stochastic factor $\eta_j$ (see \cite{bjork2009arbitrage}). Therefore, the sum of these losses $T_m^\eta$, weighted by their respective discount factors, represents the total discounted potential losses of the investment portfolio. Since assets and liabilities are often subject to different but dependent random weights (such as inflation and interest rates), analyzing their joint tail probability is essential for determining the solvency margin of an insurer (see, e.g., Chapter 9 of \cite{mcneil2015quantitative}).  
Second, in the context of quantitative economics, \eqref{eq:sum} can be applied to model systemic risk between two interconnected financial institutions (e.g., banks). Here, the two sums represent the total exposure of two banks to a series of common macroeconomic shocks. The random weights $\xi_i$ and $\eta_j$ reflect the varying sensitivity of each institution’s portfolio to these shocks over time. By investigating the joint tail probabilities of \eqref{eq:sum}, regulators can better understand the likelihood of simultaneous defaults and the potential for financial contagion (see, e.g., \cite{man2024tail}).

Though the literature mentioned above provides some valuable insights into asymptotic tail behaviors of randomly weighted sums, it is important to note that the approximations discussed therein are mostly first-order asymptotics. 
The demand for more precision in risk assessments arises from the shared aspirations of insurers and regulators. With greater accuracy, they can make more informed decisions to protect against potential risks. To the best of our knowledge, up to now, the main research on second-order expansions for the tail probability of a single randomly weighted sum $S_n^\xi$ has mainly concentrated on two types of second-order assumptions including the second-order subexponentiality (denoted by $\mathscr{S}_2$; see \cite{albrecher2010higher,lin2012seconda,lin2014second,geng2025second}) and the second-order regular
variation (denoted by $2\RV$; see \cite{degen2010risk,zhu2012tail,mao2015second,geng2024value}). 
However, most of the existing results are based on the assumption that {$\{X_i: i\in\N\}$ is a sequence of independent and identically distributed (iid) rvs and $\{X_i: i\in\N\}$ and $\{\xi_i: i\in\N\}$ are independent of each other, while this assumption of complete independence does not align with the insurance practice.}

In this present work, 
it is assumed that the primary rvs $ \{(X,Y),(X_i,Y_i): i\in \N\}$ form a sequence of real-valued iid random pairs following a common bivariate Farlie-Gumbel-Morgenstern (FGM) distribution. The FGM distribution was originally introduced by \cite{morgenstern1956einfache} and was subsequently investigated by
 \cite{gumbel1960bivariate}.
Recall that a bivariate FGM distribution function is of the form
 \begin{align}\label{eq:FGM}
 \Pi(x,y) &= F(x)G(y)\(1 +r\overline{F}(x)\overline{G}(y)\)\nn\\
&=(1+r)F(x)G(y)-rF^2(x)G(y)-rF(x)G^2(y)+rF^2(x)G^2(y),
 \end{align}
 where $F$ and $G$ are marginal distribution functions and $r$ is a real number that fulfills $|r|\leq1$ for $\Pi(\cdot,\cdot)$ to be a proper bivariate distribution function. One can see that $F^2$ and $G^2$ are also proper distributions. Clearly, the FGM distribution is flexible to model both positive and negative dependences (but not the extreme dependence). Trivially, if $r = 0$, then \eqref{eq:FGM} describes a joint distribution function of two independent rvs. We refer readers to \cite{kotz2004continuous} for a general review of the FGM distribution.

It is worth mentioning that the FGM dependence structure has wide applications in risk management and extreme value theory since it has a very nice and tractable presentation. For example, \cite{chen2011finite} and \cite{chen2015ruin} considered a discrete-time insurance risk model, in which it was assumed that the insurance risk and the financial risk follow an FGM distribution with parameters $r\in(-1,1]$ and $r=-1$, respectively; also see \cite{chen2014ruin}. In addition, \cite{yang2014asymptotic} investigated a bidimensional continuous-time risk model in which the two types of claim risk from two different lines of insurance business are assumed to follow an FGM distribution with $r\in(-1,1]$. \cite{mao2015risk} studied the second-order asymptotics of risk concentrations based on expectiles as well as $\VaR$ assuming that the loss variables fulfill a two- or high-dimensional generalized FGM copula. \cite{yang2022secondb} established a second-order expansion for the tail probability of the stochastic discounted value of the aggregate net losses under the assumption that the insurance risk and the financial risk follow an FGM distribution with parameter $r\in(-1,1)$.  

Our main contributions are as follows. First, under some mild technical conditions, we construct the second-order asymptotics of the joint tail behavior of randomly weighted sums with second-order subexponential distributions and the FGM dependence with parameter $r\in[-1,1]$. Compared with the literature discussed in the previous paragraph, the key reason that we are able to handle the entire interval $[-1,1]$ is that we carry out a second-order expansion. Second, we propose the second-order asymptotics of the sum tail behavior of randomly weighted sums under the above setting. Third, we further obtain second-order asymptotics of the joint tail probability and the sum tail probability for the stochastic discounted value of aggregate net losses in a bidimensional discrete-time risk model as an actuarial application. Finally, compared with the first-order asymptotics, some numerical results are presented to illustrate the accuracy of our second-order ones. The methods used in this paper can be applied to handle multidimensional and other types of risk models, but to avoid overly complicated settings, they are not covered in this paper but discussed in the conclusion. Note that the general model~\eqref{eq:sum} allows $n$ and $m$ to be arbitrary positive integers, while the specific actuarial application~\eqref{eq:U} specializes to $n = m$ for the natural reason that both business lines operate over the same time horizon $n$. The main theoretical results, Theorems~\ref{the:1} and~\ref{the:2}, are established for general $n, m \in \mathbb{N}$, and the risk model~\eqref{eq:U} serves as one concrete application.

The rest of this paper is organized as follows. Section \ref{sec:2} introduces some preliminaries on second-order subexponential distributions and presents two main theoretical results afterwards. 
All the proofs of the results obtained are shown in Section \ref{sec:appendix}. Section \ref{sec:3} proposes a bidimensional discrete-time risk model as an application and gives numerical simulations to illustrate the accuracy of our results in comparison with the first-order asymptotics. Section \ref{sec:conc} concludes the paper. 

\setcounter{equation}{0}
\section{Preliminaries and main results}\label{sec:2}
Throughout this paper, for any distribution function $F$, we denote its tail by $\overline F(x)=1-F(x)$. For two real numbers $a$ and $b$, write $a\vee b=\max\{a,b\}$ and $a\wedge b=\min\{a,b\}$.  All limit relationships are based on $x\rightarrow\infty$ or {$(x, y)\rightarrow (\infty,\infty)$} unless stated otherwise.  For two positive functions $f(\cdot)$ and $g(\cdot)$, write $f(x)=O(g(x))$ if $\limsup f(x)/g(x)<\infty$; write $f(x)\asymp g(x)$ {(we say that $f(x)$ is weakly equivalent to $g(x)$)} if both $f(x)=O(g(x))$ and $g(x)=O(f(x))$; write $f(x)=o(g(x))$ if $\lim f(x)/g(x)=0$; write $f(x)\lesssim g(x)$ if $\limsup{f(x)}/{g(x)}\leq 1$; write $f(x)\gtrsim g(x)$ if $\liminf{f(x)}/{g(x)}\geq 1$; write $f(x)\sim g(x)$ if $\lim{f(x)}/{g(x)}=1$. Furthermore, for two positive trivariate functions $f(\cdot,\cdot,\cdot)$ and $g(\cdot,\cdot,\cdot)$, we say that the asymptotic relation $f(x,y,t)\sim g(x,y,t)$ holds uniformly over all $t$ in a nonempty set $\Lambda$ if $$ \lim_{ {(x, y)\rightarrow (\infty,\infty)}}\sup_{t\in\Lambda} \Big|\frac{f(x,y,t)}{g(x,y,t)}-1\Big|=0.$$ By convention, for any rv $X$, write $X^+=\max\{X,0\}$. The indicator function of an event $A$ is denoted by $\mathbf{I}_A$.

\subsection{ Second-order subexponential distribution}
This paper focuses on second-order subexponential distributions that were first proposed by \cite{lin2012secondb} and were studied by many scholars subsequently. For convenience in our presentation, let us recall some related definitions and notation. Assume that $\overline F(x)>0 $ holds for all $x > 0$. For $t>0$, write $\Delta(t):=(0,t]$,
\begin{eqnarray*}
x+\Delta(t):=(x,x+t],
\end{eqnarray*}
and
\begin{eqnarray*}
F(x+\Delta(t)):=F(x+t)-F(x)~(\mbox{also~denoted~by}~F(x,x+t] ).
\end{eqnarray*}
 It is known that a distribution $F$ on $[0,+\infty)$ is classified as a distribution of the subexponential class, denoted by $\mathscr{S}$, if
\begin{eqnarray*}
{\overline {F^{2 * }}}(x) \sim 2\overline F(x),
\end{eqnarray*}
where $F^{2*}$ denotes the $2$-fold convolution of the distribution $F$ with itself. The class $\mathscr{S}$ was first introduced by \cite{chistyakov1964theorem}. For more properties of $\mathscr{S}$ and some related classes, we refer readers to \cite{embrechts2013modelling} and \cite{foss2011introduction}. In addition, a distribution $F$ on $(-\infty,\infty)$ is of the local long-tailed class $\mathscr{L}_{\Delta(t)}$ for some $t>0$, if, uniformly in $y \in [0,1]$,
\begin{eqnarray*}
F(x+y+\Delta(t))\sim F(x+\Delta(t)).
\end{eqnarray*}
Furthermore, if $F\in\mathscr{L}_{\Delta(t)}$ for all $t>0$, we denote it by $F\in \mathscr{L}_\Delta$ which was first proposed by \cite{asmussen2003applied}. Below, we give the definitions of second-order subexponential distributions supported on the nonnegative half-line and the whole real line. {This concept was proposed and developed in \cite{lin2012secondb} and \cite{yang2022seconda}, respectively.}

\begin{definition}
A distribution $F$ on $[0,\infty)$ with a finite mean $\mu_F$ is of the second-order subexponential class, denoted by $F\in\mathscr{S}_{2}$, if $F \in \mathscr{L}_{\Delta}$ and
\begin{eqnarray}\label{eq:2.1}
\overline {{F^{2 * }}} (x) - 2\overline F(x) \sim 2{\mu _F}F(x,x + 1].
\end{eqnarray}
Indeed, the relation \eqref{eq:2.1} is equivalent to 
\begin{align*}
  \frac{\overline {{F^{2 * }}} (x)}{\overline F(x)}= 2+2\mu_F \frac{F(x,x + 1]}{\overline F(x)}+o\(\frac{F(x,x + 1]}{\overline F(x)}\), 
\end{align*}
which is consistent with the first-order expansion of univariate functions from the perspective of the Taylor formula.

\end{definition}

\begin{definition}
 A function $f(\cdot)$ is said to be almost decreasing if it is positive ultimately and
 $$\sup_{y\geq x}f(y)\asymp f(x),~~~~x\rightarrow\infty.$$
\end{definition}

\begin{definition}
A distribution $F$ on $(-\infty,\infty)$ with a finite mean $\mu_F$ is of the second-order subexponential class, denoted by $\widetilde {{\mathscr S}_2}$, if $F^+ \in \mathscr{S}_{2}$ and $F(x,x+1]$ is almost decreasing, where $F^+(x)=F(x){\bf I}_{\{x\geq 0\}}$. 
\end{definition}

As stated in \cite{lin2012seconda} and \cite{yang2022seconda}, we can see that $\widetilde {{\mathscr S}_2}$ is large enough to include the following distributions, such as Pareto, Lognormal and Weibull (with the parameter between 0 and 1) distributions. {Moreover, by Lemma 5.3 of \cite{lin2014second}, we have $\overline{F}^2(x)=o\(F(x,x+1]\)$ if $F\in\widetilde {{\mathscr S}_2}$.}

\subsection{Main results }
 We use the following notations to state the main results. For any distribution $U$, denote by $\mu_U$ the {mean} of $U$. Set $\Lambda_n^i=\{1,2,\ldots,n\}\backslash\{i\}$ and $\Lambda_n^{i,j}=\{1,2,\ldots,n\}\backslash\{i,j\}$.

\begin{theorem}\label{the:1}
Let $\{(X,Y),(X_i,Y_i):i\in\N\}$ be a sequence of real-valued iid random vectors with marginal distributions $F$ and $G$, following a common bivariate FGM distribution \eqref{eq:FGM}  with parameter $r\in[-1,1]$, and let $\{\xi_i,\eta_i:  i\in \N\}$ be two sequences of arbitrarily dependent nonnegative rvs, independent of $\{(X,Y),(X_i,Y_i):i\in\N\}$, satisfying $\p(\xi_i\in[a_1,b_1],\eta_j\in[a_2,b_2])=1$ for all $i,j\in\N$ and some $0<a_1\leq b_1<\infty$, $0<a_2\leq b_2<\infty$. If ${F},{G} \in {\widetilde{\mathscr S}_2}$, then for fixed $n,m\in\N$, it holds that {
\begin{align}\label{eq:the1}
 &\quad \p\(S_n^\xi>x,T_m^\eta>y\)\nn\\
 &=\sum_{i=1}^n\sum_{j=1}^m \p(\xi_iX_i>x,\eta_jY_j>y)\nn\\
 &\quad+\sum_{i=1}^n\sum_{j=1}^m\(\sum_{l\in\Lambda_m^j}\E\[\eta_lY_l\I_{\{\xi_iX_i>x,\eta_jY_j\in(y,y+1]\}}\]+\sum_{l\in\Lambda_n^i}\E\[\xi_lX_l\I_{\{\xi_iX_i\in(x,x+1],\eta_jY_j>y\}}\]\)\nn\\
&\quad+ {\sum_{i=1}^n\sum_{j=1}^mo\(\Big(\p^2\(\xi_iX_i\in(x,x+1],\eta_jY_j>y\)+\p^2\(\xi_iX_i>x,\eta_jY_j\in(y,y+1]\)\Big)^{1/2}\)}.
\end{align}}
\end{theorem}

\begin{remark}\label{re:2.1}
The first term in the right-hand side of \eqref{eq:the1} is the first-order asymptotics of the joint tail behavior of randomly weighted sums and the remaining parts are the second-order expansion terms. 
On one hand, by Lemma 3.2 and the similar arguments in Remark 2.1 in  \cite{yang2022secondb}, it is easy to verify that the second term in the right-hand side of \eqref{eq:the1} is weakly equivalent to
$$\sum_{i=1}^n\sum_{j=1}^m\Big(\p\(\xi_iX_i>x,\eta_jY_j\in(y,y+1]\)+\p\(\xi_iX_i\in(x,x+1],\eta_jY_j>y\)\Big).$$
     On the other hand, by the simple fact that $\Delta x+\Delta y\asymp \sqrt{(\Delta x)^2+(\Delta y)^2}$ as $(\Delta x,\Delta y)\rightarrow(0+,0+)$, we have
  \begin{align*}
    &\quad \p\(\xi_iX_i>x,\eta_jY_j\in(y,y+1]\)+\p\(\xi_iX_i\in(x,x+1],\eta_jY_j>y\)\\
   & \asymp \Big(\p^2\(\xi_iX_i\in(x,x+1],\eta_jY_j>y\)+\p^2\(\xi_iX_i>x,\eta_jY_j\in(y,y+1]\)\Big)^{1/2}.
  \end{align*}
Thus, the third term in the right-hand side of \eqref{eq:the1} is also the higher-order infinitesimal of the second one, which is also consistent with the first-order Taylor expansion of bivariate functions if  both sides of the relation \eqref{eq:the1} are divided by $\p(\xi_iX_i>x,\eta_jY_j>y)$.
\end{remark}


   If we further restrict the primary rvs to belong to some smaller heavy-tailed distribution class (for instance, $F$ and $G$ have density functions belonging to the regular variation family; see Corollary \ref{cor:1} below), then the second-order asymptotic formula of the joint tail behavior of randomly weighted sums becomes more explicit. 

\begin{definition}\label{de:3}
	A measurable function $f$ valued on $(-\infty,\infty)$ is \textit{regularly varying} at infinity with index $\alpha \in \R$ if, for all $t>0$,
	\begin{align}\label{eq:de3}
	    \lim\limits_{x \rightarrow \infty} \frac{f(tx)}{f(x)} = t^\alpha,
	\end{align}
	which is denoted by $f \in \RV_{\alpha}$. 
\end{definition}


 {Based on Theorem \ref{the:1}, we have the following corollary in terms of regularly-varying functions, where the parameter condition is to ensure a finite mean.}
\begin{corollary}\label{cor:1}
Under the same conditions of Theorem \ref{the:1}, if we further assume that $F$ and $G$ have respective density functions $f\in\RV_{-(\alpha+1)}$ and $g\in\RV_{-(\beta+1)}$ for some { $\alpha,\beta>1$}, then
\begin{align}\label{eq:cor1}
 &\quad \p\(S_n^\xi>x,T_m^\eta>y\)\nn\\
&=\sum_{i=1}^n\sum_{j=1}^m\p(\xi_iX_i>x,\eta_jY_j>y)\nn\\
&\quad+\sum_{i=1}^n\sum_{j=1}^m\(\mu_{G}g(y)\sum_{l\in\Lambda_m^j}\E\[\eta_l\eta_j^\beta\I_{\{\xi_iX_i>x\}}\]+\mu_{F}f(x)\sum_{l\in\Lambda_n^i}\E\[\xi_l\xi_i^\alpha\I_{\{\eta_jY_j>y\}}\]\)\nn\\
&\quad+r\sum_{i=1}^{n\wedge m}\sum_{j\in\Lambda_m^{i}}\(\mu_{G}\E\[\eta_i\eta_j^\beta\I_{\{\xi_jX_j>x\}}\]+(\mu_{G^2}-\mu_{G})\E\[\eta_i\eta_j^\beta\I_{\{\xi_iX_i>x\}}\]\)g(y)\nn\\
&\quad+r\sum_{j=1}^{n\wedge m}\sum_{i\in\Lambda_n^j}\(\mu_{F}\E\[\xi_j\xi_i^\alpha\I_{\{\eta_iY_i>y\}}\]+(\mu_{F^2}-\mu_{F})\E\[\xi_j\xi_i^\alpha\I_{\{\eta_iY_i>y\}}\]\)f(x)\nn\\
&\quad+ {o\Big(\sqrt{x^2+y^2}f(x)g(y)\Big)}.
\end{align}
\end{corollary}

Furthermore, we consider the sum tail probability for the two randomly weighted sums.

\begin{theorem}\label{the:2}
Under the conditions of Theorem \ref{the:1}, if we further assume that $F(x,x+1]\asymp G(x,x+1]$, then for fixed $n,m\in\N$, it holds that
    \begin{align}\label{eq:the2}
   &\quad \p\(S_n^\xi+T_m^\eta>x\)\nn\\
&=\sum_{i=1}^n\p(\xi_iX_i>x)+\sum_{j=1}^m\p(\eta_jY_j>x)+\mu_F\sum_{i=1}^n\sum_{l\in\Lambda_n^i}\E\[\xi_i\I_{\{\xi_lX_l
    \in(x,x+1]\}}\] \nn\\
&\quad+\mu_G\sum_{j=1}^m\sum_{l\in\Lambda_m^j}\E\[\eta_j\I_{\{\eta_lY_l\in(x,x+1]\}}\]+\sum_{i=1}^n\sum_{j=1}^m\Big(\E\[\xi_iX_i\I_{\{\eta_jY_j\in(x,x+1]\}}\]+\E[\eta_jY_j\I_{\{\xi_iX_i\in(x,x+1]\}}]\Big)\nn\\
    &\quad+o\(\sum_{i=1}^n\p(\xi_iX_i
    \in(x,x+1])+\sum_{j=1}^m\p(\eta_jY_j
    \in(x,x+1])\).
\end{align}
\end{theorem}

{The second-order asymptotic formula of the sum tail probability for the randomly weighted sum is composed of five parts, in which the first and second parts express the first-order asymptotics, the third and fourth parts show the second-order asymptotic portion of a single randomly weighted sum, and the fifth part reveals the second-order asymptotic portion of the interaction of the two randomly weighted sums, which includes the interdependent structure.}

Similarly to the above, the following corollary can be derived from Theorem \ref{the:2}.

\begin{corollary}\label{cor:2}
    Under the conditions of Theorem \ref{the:1}, if we further assume that $F$ and $G$ have respective density {functions} $f,g\in\RV_{-(\alpha+1)}$  for some {$\alpha>1$}, then
  \begin{align}\label{eq:cor2}
    &\quad \p\(S_n^\xi+T_m^\eta>x\)\nn\\
   &=\sum_{i=1}^n\p(\xi_iX_i>x)+\sum_{j=1}^m\p(\eta_jY_j>x)+\mu_F\sum_{i=1}^n\sum_{l\in\Lambda_n^i}\E\[\xi_i\xi_l^{\alpha}\]f(x) \nn\\   
&\quad+\mu_G\sum_{j=1}^m\sum_{l\in\Lambda_m^j}\E\[\eta_j\eta_l^\alpha\]g(x)+\sum_{i=1}^n\sum_{j=1}^m\Big(\mu_F\E\[\xi_i\eta_j^\alpha\]g(x)+\mu_G\E\[\eta_j\xi_i^\alpha\]f(x)\Big)\nn\\
    &\quad+r\sum_{i=1}^{n\wedge m}\Big(\E\[\xi_i\eta_i^\alpha\](\mu_{F^2}-\mu_F)g(x)+\E\[\eta_i\xi_i^\alpha\](\mu_{G^2}-\mu_G)f(x)\Big)+o\(f(x)\).
\end{align}
\end{corollary}

\setcounter{equation}{0}\par
\section{Proofs of the main results}\label{sec:appendix}

In this section, we first prepare two lemmas and the proof of Theorem \ref{the:1} is given afterwards. The following Lemma \ref{lem:1}, which may be of its own interest, is 
a non-identically distributed version of Theorem 3.1 in \cite{lin2020second}. Lemma \ref{lem:2} plays an important role in the proof of Theorem \ref{the:1}.
Next, we present Lemmas \ref{lem:3} and \ref{lem:4} and the proof of Theorem \ref{the:2} is given at the end of this paper.
For convenience in writing, we use the following symbols:
$$B_{11}(x):=\overline{F}\(\frac{x}{c_1}\)+\overline{F}\(\frac{x}{c_2}\)+c_2\mu_{F}F\(\frac{x}{c_1},\frac{x+1}{c_1}\](1+o(1))+c_1\mu_{F}F\(\frac{x}{c_2},\frac{x+1}{c_2}\](1+o(1));$$
$$B_{12}(x):=\overline{F}\(\frac{x}{c_1}\)+2\overline{F}\(\frac{x}{c_2}\)+c_2\mu_{F^2}F\(\frac{x}{c_1},\frac{x+1}{c_1}\](1+o(1))+2c_1\mu_{F}F\(\frac{x}{c_2},\frac{x+1}{c_2}\](1+o(1));$$
$$B_{21}(x):=2\overline{F}\(\frac{x}{c_1}\)+\overline{F}\(\frac{x}{c_2}\)+2c_2\mu_{F}F\(\frac{x}{c_1},\frac{x+1}{c_1}\](1+o(1))+c_1\mu_{F^2}F\(\frac{x}{c_2},\frac{x+1}{c_2}\](1+o(1));$$
$$B_{22}(x):=2\overline{F}\(\frac{x}{c_1}\)+2\overline{F}\(\frac{x}{c_2}\)+2c_2\mu_{F^2}F\(\frac{x}{c_1},\frac{x+1}{c_1}\](1+o(1))+2c_1\mu_{F^2}F\(\frac{x}{c_2},\frac{x+1}{c_2}\](1+o(1));$$
$$D_{11}(y):=\overline{G}\(\frac{y}{d_1}\)+\overline{G}\(\frac{y}{d_2}\)+d_2\mu_{G}G\(\frac{y}{d_1},\frac{y+1}{d_1}\](1+o(1))+d_1\mu_{G}G\(\frac{y}{d_2},\frac{y+1}{d_2}\](1+o(1));$$
$$D_{12}(y):=\overline{G}\(\frac{y}{d_1}\)+2\overline{G}\(\frac{y}{d_2}\)+d_2\mu_{G^2}G\(\frac{y}{d_1},\frac{y+1}{d_1}\](1+o(1))+2d_1\mu_{G}G\(\frac{y}{d_2},\frac{y+1}{d_2}\](1+o(1));$$
$$D_{21}(y):=2\overline{G}\(\frac{y}{d_1}\)+\overline{G}\(\frac{y}{d_2}\)+2d_2\mu_{G}G\(\frac{y}{d_1},\frac{y+1}{d_1}\](1+o(1))+d_1\mu_{G^2}G\(\frac{y}{d_2},\frac{y+1}{d_2}\](1+o(1));$$
$$D_{22}(y):=2\overline{G}\(\frac{y}{d_1}\)+2\overline{G}\(\frac{y}{d_2}\)+2d_2\mu_{G^2}G\(\frac{y}{d_1},\frac{y+1}{d_1}\](1+o(1))+2d_1\mu_{G^2}G\(\frac{y}{d_2},\frac{y+1}{d_2}\](1+o(1)).$$

\begin{lemma}\label{lem:1}
Let $X_1$ and $X_2$ be two independent real-valued rvs with respective distributions $F_1$ and $F_2$. If ${F_1} \in {\widetilde{\mathscr S}_2} $ and there exist constants $K>0,A\in\R$ such that
\begin{align}\label{eq:3.1}
\lim_{x\rightarrow\infty}\frac{\overline{F_2}(x)-K\overline{F_1}(x)}{F_1(x,x+1]}=A,
\end{align}
then we have ${F_2} \in \widetilde{\mathscr S}_2 $. In addition, for any fixed $0<a\leq b<\infty$, the relation
\begin{align*}
\p\(c_1X_1+c_2X_2>x\)
&=\overline{F_1}\(\frac{x}{c_1}\)+K\overline{F_1}\(\frac{x}{c_2}\)+(A+Kc_1\mu_{F_1})\p(c_2X_1\in(x,x+1])(1+o(1))\nn\\
&\quad+c_2\mu_{F_2}\p(c_1X_1\in(x,x+1])(1+o(1))\nn\\
&=\sum_{i=1}^2\(\p(c_iX_i>x)+c_i\mu_{F_i}\sum_{j\neq i}^2\p(c_jX_j\in(x,x+1])(1+o(1))\)
\end{align*}
holds uniformly for $(c_1,c_2)\in[a,b]^2$.
\end{lemma}
\begin{proof} First, the relation \eqref{eq:3.1} is equivalent to
\begin{align*}
 \overline F_2(x)=K\overline F_1(x)+AF_1(x,x+1]+o\big(F_1(x,x+1]\big).
 \end{align*}
Thus, according to Proposition 2.3 of \cite{lin2012secondb}, we know that $F_2^+\in \mathscr S_2$. On the other hand, from the relation \eqref{eq:3.1}, it is easy to see
\begin{align*}
  F_2(x,x+1]=K F_1(x,x+1]+o\big(F_1(x,x+1]\big),
 \end{align*}
 which implies that $F_2(x,x+1]$ is also almost decreasing as $x\rightarrow\infty$, and hence we have ${F_2} \in \widetilde{\mathscr S}_2$.
 
Next, for any $0<\epsilon<1$ and any fixed $\delta>0$, according to \eqref{eq:3.1} and Lemma 4.4  of \cite{lin2020second},  there exist constants $C_1,\omega_0 \geq 0$ such that for all $t\geq \delta, \tau\geq 1$ and $x>w_0$, 
 \begin{align}\label{eq:3.2}
 (A-\epsilon)F_1(x,x+1]\leq \overline{F_2}(x)-K\overline{F_1}(x)\leq (A+\epsilon)F_1(x,x+1],
 \end{align}
 and
 \begin{align}\label{eq:3.3}
    F_i(x,x+\tau t]\leq C_1\tau F_i(x,x+t],~~~~~i=1,2.
 \end{align}
Furthermore, it is clear that
 \begin{align}\label{eq:sum1}
 \p\(c_1X_1+c_2X_2>x\)&=\int_{-\infty}^{\infty}\p(c_2X_2>x-t)\p(c_1X_1\in\d t)\nn\\
 &=\overline{F_2}\(\frac{x}{c_2}\)+\int_{-\infty}^{\infty}\(\overline{F_2}\(\frac{x-t}{c_2}\)-\overline{F_2}\(\frac{x}{c_2}\)\)\p(c_1X_1\in\d t).
 \end{align}
Now we deal with the integration term in the right-hand side of \eqref{eq:sum1}. For all {$w > w_0 \geq0$ and} $x>2\omega>2b\omega_0$, we divide it into four terms as
 \begin{align}\label{eq:J}
\int_{-\infty}^{-w}+\int_{-w}^{w}+\int_{w}^{x-w}+\int_{x-w}^{\infty}:=J_1+J_2+J_3+J_4.
 \end{align}
 To deal with $J_1$, by \eqref{eq:3.3}, for all $(c_1,c_2)\in[a,b]^2$, we have
\begin{align*}
J_1&=\int_{-\infty}^{-w}F_2\(\frac{x}{c_2},\frac{x-t}{c_2}\]\p(c_1X_1\in\d t)\leq C_1F_2\(\frac{x}{c_2},\frac{x+1}{c_2}\]\int_{-\infty}^{-w}|t|\p(c_1X_1\in\d t).\nn
\end{align*}
 For $J_2$, due to Lemma 4.8 of \cite{lin2020second}, as $x\rightarrow\infty$, we have, uniformly for $(c_1,c_2)\in[a,b]^2$,
\begin{align*}
 J_2&= \p(c_2X_2\in(x,x+1])\int_{-w}^wt\p(c_1X_1\in\d t)(1+o(1)).
 \end{align*} 
Since $\mu_{F_1}<\infty$, as $w\rightarrow\infty$, we have
\begin{align}\label{eq:J1}
   J_1=o\(\p(c_2X_2\in(x,x+1])\), 
\end{align}
and
\begin{align}\label{eq:J2}
J_2=c_1\mu_{F_1}\p(c_2X_2\in(x,x+1])=Kc_1\mu_{F_1}\p(c_2X_1\in(x,x+1])(1+o(1)).
\end{align}
{For $J_3$, by \eqref{eq:3.2} and \eqref{eq:3.3},  as $x\rightarrow\infty$,  we obtain that, uniformly for $(c_1,c_2)\in[a,b]^2$,
 \begin{align*}
J_3&=\int_w^{x-w}\(\(\overline{F_2}\(\frac{x-t}{c_2}\)-K\overline{F_1}\(\frac{x-t}{c_2}\)\)-\(\overline{F_2}\(\frac{x}{c_2}\)-K\overline{F_1}\(\frac{x}{c_2}\)\)\)\p(c_1X_1\in\d t)\\
 &\quad {+}K\int_w^{x-w}\(\overline{F_1}\(\frac{x-t}{c_2}\)-\overline{F_1}\(\frac{x}{c_2}\)\)\p(c_1X_1\in\d t)\\
  &\leq \int_w^{x-w}\((A+\epsilon)F_1\(\frac{x-t}{c_2},\frac{x-t}{c_2}+1\right]-(A-\epsilon)F_1\(\frac{x}{c_2},\frac{x}{c_2}+1\right]\)\p(c_1X_1\in\d t)\\
 &\quad +K\int_w^{x-w}F_1\(\frac{x-t}{c_2},\frac{x}{c_2}\right]\p(c_1X_1\in\d t)\\
 &\leq \int_w^{x-w}\((|A|+\epsilon)  F_1\(\frac{x-t}{c_2},\frac{x}{c_2}+1\right] {+}KF_1\(\frac{x-t}{c_2},\frac{x}{c_2}\]\)\p(c_1X_1\in\d t)\\
 &\leq \(C_1\(1+\frac{b}{w}\)(|A|+\epsilon)+K\)\int_w^{x-w}F_1\(\frac{x-t}{c_2},\frac{x}{c_2}\]\p(c_1X_1\in\d t).
 \end{align*}
According to the same treatments of $J_3$ of Lemma 4.11 in \cite{lin2020second}, by first letting $x\rightarrow\infty$ and then  $w\rightarrow\infty$,  it holds uniformly for all $(c_1,c_2)\in[a,b]^2$ that
\begin{align}\label{eq:J3}
   J_3= o\(\p(c_1X_1\in(x,x+1])+\p(c_2X_2\in(x,x+1])\).
\end{align}
 To deal with $J_4$, 
 for all $(c_1,c_2)\in[a,b]^2$, we have
 \begin{align*}
J_4&=\p\(c_1X_1+c_2X_2>x,c_1X_1>x-w\)-\overline{F_1}\(\frac{x-w}{c_1}\)\overline{F_2}\(\frac{x}{c_2}\)\nn\\
&=\overline{F_1}\(\frac{x}{c_1}\)+\(\int_{-\infty}^{-w}+\int_{-w}^w\)\(\overline{F_1}\(\frac{x-t}{c_1}\)-  \overline{F_1}\(\frac{x}{c_1}\)\)\p(c_2X_2\in\d t)\nn\\
&\quad+ {F_1}\(\frac{x-w}{c_1},\frac{x}{c_1}\]\overline{F_2}\(\frac{w}{c_2}\)-\overline{F_1}\(\frac{x-w}{c_1}\)\overline{F_2}\(\frac{x}{c_2}\)\nn\\  
&:=\overline{F_1}\(\frac{x}{c_1}\)+J_{41}+J_{42}+J_{43}-J_{44}.
 \end{align*}
 For $J_{41}$ and $J_{42}$, using similar treatments in $J_1$ and $J_2$ and noting $\mu_{F_2}<\infty$, by first letting $x\rightarrow\infty$ and then  $w\rightarrow\infty$, it holds uniformly for all $(c_1,c_2)\in[a,b]^2$ that
 \begin{align*}
    J_{41}=o\(\p(c_1X_1\in(x,x+1])\),
 \end{align*}
 and
 \begin{align*}
    J_{42}= c_2\mu_{F_2}\p(c_1X_1\in(x,x+1])(1+o(1)).
 \end{align*}
 For $J_{43}$, by Lemma 4.8 of \cite{lin2020second}, as $x\rightarrow\infty$, we have uniformly for all $(c_1,c_2)\in[a,b]^2$, 
 \begin{align*}
   J_{43}=w\p(c_1X_1\in(x,x+1])\overline{F_2}\(\frac{w}{c_2}\)(1+o(1)),
 \end{align*}
which further implies that, uniformly for all $(c_1,c_2)\in[a,b]^2$, 
 \begin{align*}
   J_{43}=o\(\p(c_1X_1\in(x,x+1])\),
 \end{align*}
by letting  $w\rightarrow\infty$ together with $\mu_{F_2}<\infty$.  For $J_{44}$, according to $F_1,F_2\in \widetilde{\mathscr S}_2$ and \eqref{eq:3.3}, as $x\rightarrow\infty$, it holds uniformly for all $(c_1,c_2)\in[a,b]^2$ that
 \begin{align*}
   J_{44}&\leq \frac{1}{2}\(\overline{F_1}^2\(\frac{x-w}{c_1}\)+\overline{F_2}^2\(\frac{x}{c_2}\)\)\\
   &=o\(F_1\(\frac{x-w}{c_1},\frac{x-w}{c_1}+1\]+F_2\(\frac{x}{c_2},\frac{x}{c_2}+1\]\)\\
   &=o\(F_1\(\frac{x}{c_1},\frac{x+1}{c_1}\]+F_2\(\frac{x}{c_2},\frac{x+1}{c_2}\]\).
 \end{align*}
Thus, we have uniformly for all $(c_1,c_2)\in[a,b]^2$, 
\begin{align}\label{eq:J4}
   J_4=\overline{F_1}\(\frac{x}{c_1}\)+  c_2\mu_{F_2}\p(c_1X_1\in(x,x+1])(1+o(1))+o\(\p(c_2X_2\in(x,x+1])\).
\end{align}
 Substituting \eqref{eq:J1}-\eqref{eq:J4} into \eqref{eq:J} and \eqref{eq:sum1} yields the desired result.}
    \end{proof}
 


 %


\begin{lemma}\label{lem:2}
Let $\{(X,Y),(X_i,Y_i):i\in\N\}$ be a sequence of real-valued iid random vectors with marginal distributions $F$ and $G$, following a common bivariate FGM distribution \eqref{eq:FGM}  with parameter $r\in[-1,1]$. If ${F,G} \in {\widetilde{\mathscr S}_2}$, then for any fixed $n,m\in\N$ and some $0<a\leq b<\infty$, it holds uniformly for $(\underline{c_n},\underline{d_m}):=(c_1,\ldots,c_n,d_1,\ldots,d_m)\in[a,b]^{n+m}$ that
\begin{align}\label{eq:sum2}
 &\quad \p\(S_n^c>x,T_m^d>y\)\nn\\
&=\sum_{i=1}^n\sum_{j=1}^m\p(c_iX_i>x,d_jY_j>y)\nn\\
&\quad+\sum_{i=1}^n\sum_{j=1}^m\(\sum_{l\in\Lambda_m^j}d_l\E\[Y_l\I_{\{c_iX_i>x,d_jY_j\in(y,y+1]\}}\]+\sum_{l\in\Lambda_n^i}c_l\E\[X_l\I_{\{c_iX_i\in(x,x+1],d_jY_j>y\}}\]\)\nn\\
&\quad+ {\sum_{i=1}^n\sum_{j=1}^mo\(\(\overline{F}^2\(\frac{x}{c_i}\)G^2\(\frac{y}{d_j},\frac{y+1}{d_j}\]+F^2\(\frac{x}{c_i},\frac{x+1}{c_i}\]\overline{G}^2\(\frac{y}{d_j}\)\)^{1/2}\)}.
\end{align}
\end{lemma}
\begin{proof}
Since the desired result is trivial for $n=m=1$, it suffices to prove the case of $n,m\geq 2$. Due to the complexity of the presentation, we only present the proof process for the case of $n=m=2$ and the extension to $n,m> 2$ can follow from the same logic. 
According to \eqref{eq:FGM}, for $r\in[-1,1]$, the bivariate distribution function $\Pi(\cdot,\cdot)$ can be decomposed into four terms:
\begin{align}\label{eq:decom1}
 \Pi=(1+r)FG-rF^2G-rFG^2+rF^2G^2.
\end{align}
In particular, when $r=-1$, the first term in \eqref{eq:decom1} vanishes. Apart from the absence of this term, the subsequent argument remains unchanged; the key reason is that we are doing a second-order expansion. Therefore, to avoid repetition, we do not present a separate proof for the case $r=-1$; instead, we incorporate it into the proof for $r\in(-1,1]$.


According to the decomposition \eqref{eq:decom1}, we introduce four independent rvs $X_1',X_1^*,Y_1'$ and $Y_1^*$, independent of $(X_2,Y_2)$, with corresponding distribution functions $F,F^2,G$ and $G^2$. Thus, we split the following probability into four parts as
\begin{align*}
\p\(c_1X_1+c_2X_2>x,d_1Y_1+d_2Y_2>y\)=\sum_{i=1}^4I_{i}(x,y),
\end{align*}
with
\begin{align*}
\left\{
\begin{array}
[c]{lll}%
I_{1}(x,y)=(1+r) \p\(c_1X'_1+c_2X_2>x,d_1Y'_1+d_2Y_2>y\),\\
I_{2}(x,y)=-r \p\(c_1X^*_1+c_2X_2>x,d_1Y'_1+d_2Y_2>y\),\\
I_{3}(x,y)=-r \p\(c_1X'_1+c_2X_2>x,d_1Y^*_1+d_2Y_2>y\),\\
I_{4}(x,y)=r \p\(c_1X^*_1+c_2X_2>x,d_1Y^*_1+d_2Y_2>y\).
\end{array}
\right.
\end{align*}
 Then further, let $(X_2',X_2^*,Y_2',Y_2^*)$ be an independent copy of $(X_1',X_1^*,Y_1',Y_1^*)$ and independent of all the random sources mentioned above. Then, $I_{1}(x,y)$ can be further decomposed into four parts as follows:
 \begin{align*}
 I_{1}(x,y)=(1+r)\sum_{j=1}^4I_{1j}(x,y),
 \end{align*}
with
\begin{align*}
\left\{
\begin{array}
[c]{lll}%
I_{11}(x,y)=(1+r) \p\(c_1X'_1+c_2X'_2>x\)\p\(d_1Y'_1+d_2Y'_2>y\),\\
I_{12}(x,y)=-r \p\(c_1X'_1+c_2X^*_2>x\)\p\(d_1Y'_1+d_2Y'_2>y\),\\
I_{13}(x,y)=-r \p\(c_1X'_1+c_2X'_2>x\)\p\(d_1Y'_1+d_2Y^*_2>y\),\\
I_{14}(x,y)=r \p\(c_1X'_1+c_2X^*_2>x\)\p\(d_1Y'_1+d_2Y^*_2>y\).\\
\end{array}
\right.
\end{align*}
 Due to $F,G \in {\widetilde{\mathscr S}_2}$, it is obvious that $$\lim_{x\rightarrow\infty}\frac{\overline {F^2}(x)-2\overline{F}(x)}{F(x,x+1]}=\lim_{x\rightarrow\infty}\frac{\overline{F}(x)(2-\overline{F}(x))-2\overline{F}(x)}{F(x,x+1]}=0,$$
and
 $$\lim_{x\rightarrow\infty}\frac{\overline {G^2}(x)-2\overline{G}(x)}{G(x,x+1]}=\lim_{x\rightarrow\infty}\frac{\overline{G}(x)(2-\overline{G}(x))-2\overline{G}(x)}{G(x,x+1]}=0.$$
Applying Lemma \ref{lem:1} with $K=2$ and $A=0$ for $B_{12},B_{21},D_{12},D_{21}$, it holds uniformly for $(\underline{c_2},\underline{d_2})\in[a,b]^4$ that
\begin{align*}
\left\{
\begin{array}
[c]{lll}%
I_{11}(x,y)=(1+r)B_{11}(x)D_{11}(y),\\
I_{12}(x,y)=-r B_{12}(x)D_{11}(y),\\
I_{13}(x,y)=-r B_{11}(x)D_{12}(y),\\
I_{14}(x,y)=r B_{12}(x)D_{12}(y).\\
\end{array}
\right.
\end{align*}
Hence, we have, uniformly for $(\underline{c_2},\underline{d_2})\in[a,b]^4$,
$$I_{1}(x,y)=(1+r)\Big(B_{11}(x)D_{11}(y)+r\big(B_{12}(x)-B_{11}(x)\big)\big(D_{12}(y)-D_{11}(y)\big)\Big).$$
By the same procedures as above, we can derive that, uniformly for $(\underline{c_2},\underline{d_2})\in[a,b]^4$,
\begin{align*}
\left\{
\begin{array}
[c]{lll}%
I_{2}(x,y)=-r \Big(B_{21}(x)D_{11}(y)+r\big(B_{22}(x)-B_{21}(x)\big)\big(D_{12}(y)-D_{11}(y)\big)\Big),\\
I_{3}(x,y)=-r \Big(B_{11}(x)D_{21}(y)+r\big(B_{12}(x)-B_{11}(x)\big)\big(D_{22}(y)-D_{21}(y)\big)\Big),\\
I_{4}(x,y)=r \Big(B_{21}(x)D_{21}(y)+r\big(B_{22}(x)-B_{21}(x)\big)\big(D_{22}(y)-D_{21}(y)\big)\Big).\\
\end{array}
\right.
\end{align*}
 Combining all these results, we obtain that, uniformly for $(\underline{c_2},\underline{d_2})\in[a,b]^4$,
\begin{align*}
&\quad\sum_{i=1}^4I_{i}(x,y)\\
&=B_{11}(x)D_{11}(y)+r\big(B_{12}(x)-B_{11}(x)\big)\big(D_{12}(y)-D_{11}(y)\big)\\
&\quad +r\big(B_{21}(x)-B_{11}(x)\big)\big(D_{21}(y)-D_{11}(y)\big)\nn\\
&\quad+r^2\big(B_{22}(x)-B_{21}(x)-B_{12}(x)+B_{11}(x)\big)\big(D_{22}(y)-D_{21}(y)-D_{12}(y)+D_{11}(y)\big).
\end{align*}

For $B_{11}(x)D_{11}(y)$, it follows that, uniformly for $(\underline{c_2},\underline{d_2})\in[a,b]^4$,
\begin{align*}
 B_{11}(x)D_{11}(y)&=\sum_{i=1}^2\sum_{j=1}^2\Bigg(\overline{F}\(\frac{x}{c_i}\)\overline{G}\(\frac{y}{d_j}\)+\overline{F}\(\frac{x}{c_i}\)\sum_{l\in\Lambda_2^j}d_l\mu_{G}G\(\frac{y}{d_j},\frac{y+1}{d_j}\](1+o(1))\\
&\quad+\overline{G}\(\frac{y}{d_j}\)\sum_{l\in\Lambda_2^i}c_l\mu_{F}F\(\frac{x}{c_i},\frac{x+1}{c_i}\](1+o(1))\Bigg).
\end{align*}
For $\big(B_{12}(x)-B_{11}(x)\big)\big(D_{12}(y)-D_{11}(y)\big)$, we have that, uniformly for $(\underline{c_2},\underline{d_2})\in[a,b]^4$,
\begin{align*}
  &\quad \big(B_{12}(x)-B_{11}(x)\big)\big(D_{12}(y)-D_{11}(y)\big) \\
  &=\(\overline{F}\(\frac{x}{c_2}\)+c_2(\mu_{F^2}-\mu_F)F\(\frac{x}{c_1},\frac{x+1}{c_1}\](1+o(1))+c_1\mu_{F}F\(\frac{x}{c_2},\frac{x+1}{c_2}\](1+o(1))\)\\
  &\quad \times\(\overline{G}\(\frac{y}{d_2}\)+d_2\(\mu_{G^2}-\mu_G\)G\(\frac{y}{d_1},\frac{y+1}{d_1}\](1+o(1))+d_1\mu_{G}G\(\frac{y}{d_2},\frac{y+1}{d_2}\](1+o(1))\)\\
  &=\overline{F}\(\frac{x}{c_2}\)\(\overline{G}\(\frac{y}{d_2}\)+d_2\(\mu_{G^2}-\mu_G\)G\(\frac{y}{d_1},\frac{y+1}{d_1}\](1+o(1))+d_1\mu_{G}G\(\frac{y}{d_2},\frac{y+1}{d_2}\](1+o(1))\)\\
  &\quad+\overline{G}\(\frac{y}{d_2}\)\(c_2(\mu_{F^2}-\mu_F)F\(\frac{x}{c_1},\frac{x+1}{c_1}\](1+o(1))+c_1\mu_{F}F\(\frac{x}{c_2},\frac{x+1}{c_2}\](1+o(1))\).
\end{align*}
Similarly, it follows that, uniformly for $(\underline{c_2},\underline{d_2})\in[a,b]^4$,
\begin{align*}
  &\quad \big(B_{21}(x)-B_{11}(x)\big)\big(D_{21}(y)-D_{11}(y)\big) \\
  &=\(\overline{F}\(\frac{x}{c_1}\)+c_1(\mu_{F^2}-\mu_F)F\(\frac{x}{c_2},\frac{x+1}{c_2}\](1+o(1))+c_2\mu_{F}F\(\frac{x}{c_1},\frac{x+1}{c_1}\](1+o(1))\)\\
  &\quad \times\(\overline{G}\(\frac{y}{d_1}\)+d_1\(\mu_{G^2}-\mu_G\)G\(\frac{y}{d_2},\frac{y+1}{d_2}\](1+o(1))+d_2\mu_{G}G\(\frac{y}{d_1},\frac{y+1}{d_1}\](1+o(1))\)\\
  &=\overline{F}\(\frac{x}{c_1}\)\(\overline{G}\(\frac{y}{d_1}\)+d_1\(\mu_{G^2}-\mu_G\)G\(\frac{y}{d_2},\frac{y+1}{d_2}\](1+o(1))+d_2\mu_{G}G\(\frac{y}{d_1},\frac{y+1}{d_1}\](1+o(1))\)\\
  &\quad+\overline{G}\(\frac{y}{d_1}\)\(c_1(\mu_{F^2}-\mu_F)F\(\frac{x}{c_2},\frac{x+1}{c_2}\](1+o(1))+c_2\mu_{F}F\(\frac{x}{c_1},\frac{x+1}{c_1}\](1+o(1))\),
\end{align*}
and
\begin{align*}
 &\quad \big(B_{22}(x)-B_{21}(x)-B_{12}(x)+B_{11}(x)\big)\big(D_{22}(y)-D_{21}(y)-D_{12}(y)+D_{11}(y)\big)\\
 &=\(\sum_{1\leq i\neq j\leq 2}c_i(\mu_{F^2}-\mu_F)F\(\frac{x}{c_j},\frac{x+1}{c_j}\]\)\(\sum_{1\leq i\neq j\leq 2}d_i\(\mu_{G^2}-\mu_G\)G\(\frac{y}{d_j},\frac{y+1}{d_j}\]\)(1+o(1))\\
 &=\sum_{i=1}^2\sum_{j=1}^2 o\(\overline{F}\(\frac{x}{c_i}\)G\(\frac{y}{d_j},\frac{y+1}{d_j}\]+F\(\frac{x}{c_i},\frac{x+1}{c_i}\]\overline{G}\(\frac{y}{d_j}\)\).
\end{align*}
Thus, we have that, uniformly for $(\underline{c_2},\underline{d_2})\in[a,b]^4$,
\begin{align*}
&\quad\sum_{i=1}^4I_{i}(x,y)\\
&=\sum_{i=1}^2\sum_{j=1}^2\Bigg(\overline{F}\(\frac{x}{c_i}\)\overline{G}\(\frac{y}{d_j}\)+\overline{F}\(\frac{x}{c_i}\)\sum_{l\in\Lambda_2^j}d_l\mu_{G}G\(\frac{y}{d_j},\frac{y+1}{d_j}\]\\
&\quad+\overline{G}\(\frac{y}{d_j}\)\sum_{l\in\Lambda_2^i}c_l\mu_{F}F\(\frac{x}{c_i},\frac{x+1}{c_i}\]\Bigg)+r\sum_{i=1}^2\overline{F}\(\frac{x}{c_i}\)\overline{G}\(\frac{y}{d_i}\)\\
&\quad+r\sum_{1\leq i\neq j\leq2}\Bigg(d_i\mu_{G}\overline{F}\(\frac{x}{c_j}\)G\(\frac{y}{d_j},\frac{y+1}{d_j}\]+c_j\mu_{F}\overline{G}\(\frac{y}{d_i}\)F\(\frac{x}{c_i},\frac{x+1}{c_i}\]\nn\\
&\quad\quad+d_i(\mu_{G^2}-\mu_{G})\overline{F}\(\frac{x}{c_i}\)G\(\frac{y}{d_j},\frac{y+1}{d_j}\]+c_j(\mu_{F^2}-\mu_{F})\overline{G}\(\frac{y}{d_j}\)F\(\frac{x}{c_i},\frac{x+1}{c_i}\]\Bigg)\nn\\
&\quad+\sum_{i=1}^2\sum_{j=1}^2 o\(\(\overline{F}\(\frac{x}{c_i}\)G\(\frac{y}{d_j},\frac{y+1}{d_j}\]+F\(\frac{x}{c_i},\frac{x+1}{c_i}\]\overline{G}\(\frac{y}{d_j}\)\)\)\\
&=\sum_{i=1}^2\sum_{j=1}^2\p(c_iX_i>x,d_jY_j>y)\nn\\
&\quad+\sum_{i=1}^2\sum_{j=1}^2\(\sum_{l\in\Lambda_2^j}d_l\E\[Y_l\I_{\{c_iX_i>x,d_jY_j\in(y,y+1]\}}\]+\sum_{l\in\Lambda_2^i}c_l\E\[X_l\I_{\{c_iX_i\in(x,x+1],d_jY_j>y\}}\]\)\nn\\
&\quad+\sum_{i=1}^2\sum_{j=1}^2o\(\(\overline{F}^2\(\frac{x}{c_i}\)G^2\(\frac{y}{d_j},\frac{y+1}{d_j}\]+F^2\(\frac{x}{c_i},\frac{x+1}{c_i}\]\overline{G}^2\(\frac{y}{d_j}\)\)^{1/2}\),
\end{align*}
where in the third step, we used Remark \ref{re:2.1}. Therefore, the proof of Theorem \ref{the:1} with the case of $n,m=2$  ends.

Finally, using methods analogous to those established earlier, we can extend the result to the case of $n=m>2$. In addition, if $n, m> 2$ but $n\neq m$, then, without loss of generality, we can assume that $m>n\geq 2$. Since sets $\{Y_j: j=n+1,\ldots,m\}$ and $\{(X_i,Y_i): 1\leq i\leq n\}$ are independent, the proof of the lemma is complete.
\end{proof}

\begin{lemma}\label{lem:3}
    Let $X$ and $Y$ be two real-value rvs with respective distribution functions $F$ and $G$ satisfying a bivariate FGM distribution \eqref{eq:FGM} with parameter $r\in[-1,1]$. If $F,G \in \widetilde{\mathscr S}_2$ and $F(x,x+1]\asymp G(x,x+1]$, then, for any $ 0 < a < b < \infty $, it holds uniformly for $(c,d)\in[a,b]^{2}$ that
\begin{align}\label{lem:3.1}
        \p\(cX+dY>x\)&=\overline{F}\(\frac{x}{c}\)+\overline{G}\(\frac{x}{d}\)+d(r\mu_{G^2}+(1-r)\mu_G)F\(\frac{x}{c},\frac{x+1}{c}\](1+o(1))\nn\\
     &\quad+c(r\mu_{F^2}+(1-r)\mu_F)G\(\frac{x}{d},\frac{x+1}{d}\](1+o(1)),
\end{align}
and further
\begin{align*}
         \p\(cX+dY\in(x,x+1]\)=\(F\(\frac{x}{c},\frac{x+1}{c}\]+G\(\frac{x}{d},\frac{x+1}{d}\]\)(1+o(1)).
\end{align*}
\end{lemma}
\begin{proof} Similar to the proof of Lemma \ref{lem:2}, we first consider $r\in(-1,1]$ and introduce four independent rvs $X',X^*,Y'$ and $Y^*$ with corresponding distribution functions $F,F^2,G$ and $G^2$. According to Lemma \ref{lem:1} and $F(x,x+1]\asymp G(x,x+1]$, it holds uniformly for $(c,d)\in[a,b]^{2}$ that
\begin{align*}
\p\(cX+dY>x\)
     &=(1+r)\int_{\{u+v>x\}}\d F\( \frac{u}{c}\)\d G\(\frac{v}{d}\)-r\int_{\{u+v>x\}}\d F^2\( \frac{u}{c}\)\d G\(\frac{v}{d}\)\nn\\
     &\quad-r\int_{\{u+v>x\}}\d F^2\( \frac{u}{c}\)\d G\(\frac{v}{d}\)+r\d F^2\( \frac{u}{c}\)\d G^2\(\frac{v}{d}\)\nn\\
     &=(1+r)\p(cX'+dY'>x)-r\p(cX^*+dY'>x)\nn\\
     &\quad-r\p(cX'+dY^*>x)+r\p(cX^*+dY^*>x)\nn\\
     \quad\quad\quad\quad &=\overline{F}\(\frac{x}{c}\)+\overline{G}\(\frac{x}{d}\)+d(r\mu_{G^2}+(1-r)\mu_G)F\(\frac{x}{c},\frac{x+1}{c}\](1+o(1))\nn\\
     &\quad+c(r\mu_{F^2}+(1-r)\mu_F)G\(\frac{x}{d},\frac{x+1}{d}\](1+o(1)).
\end{align*}
Next, due to \eqref{lem:3.1} and  $F,G \in \widetilde{\mathscr S}_2 \subseteq \mathscr{L}_{\Delta} $, it holds uniformly for $(c,d)\in[a,b]^{2}$ that,
\begin{align*}
    \p\(cX+dY\in(x,x+1]\)&=\p\(cX+dY>x\)-\p\(cX+dY>x+1\)\nn\\
    &=\(F\(\frac{x}{c},\frac{x+1}{c}\]+G\(\frac{x}{d},\frac{x+1}{d}\]\)(1+o(1)).
\end{align*}
 For the case of $r=-1$, we can follow the proof idea of Lemma \ref{lem:2} to obtain the desired result. Thus, we complete this proof of the lemma.
\end{proof}

\begin{lemma}\label{lem:4}
    Under the same conditions of Lemma \ref{lem:2}, if we further assume that  $F(x,x+1]\asymp G(x,x+1]$, then for any fixed $n,m\in\N$ and any  $ 0 < a < b < \infty $, it holds uniformly for $(\underline{c_n},\underline{d_m})\in[a,b]^{n+m}$ that
    \begin{align}\label{eq:lem4.1}
        \p\(S_n^c+T_m^d>x\)&=\sum_{i=1}^n\overline{F}\(\frac{x}{c_i}\)+\sum_{1\leq i\neq l\leq n }c_i\mu_FF\(\frac{x}{c_l},\frac{x+1}{c_l}\]+\sum_{1\leq j\neq l\leq m }d_j\mu_GG\(\frac{x}{d_l},\frac{x+1}{d_l}\]\nn\\
        &\quad+\sum_{j=1}^m\overline{G}\(\frac{x}{d_j}\)+\sum_{i=1}^n\sum_{j=1}^m\(c_i\mu_FG\(\frac{x}{d_j},\frac{x+1}{d_j}\]+d_j\mu_GF\(\frac{x}{c_i},\frac{x+1}{c_i}\]\)\nn\\
        &\quad+r\sum_{i=1}^{n\wedge m}\(c_i(\mu_{F^2}-\mu_F)G\(\frac{x}{d_i},\frac{x+1}{d_i}\]+d_i(\mu_{G^2}-\mu_G)F\(\frac{x}{c_i},\frac{x+1}{c_i}\)\)\nn\\
        &\quad+o\(\sum_{i=1}^nF\(\frac{x}{c_i},\frac{x+1}{c_i}\]+\sum_{j=1}^mG\(\frac{x}{d_j},\frac{x+1}{d_j}\]\).
    \end{align}
\end{lemma}
\begin{proof}
First,  for notational convenience, denote $Z_1:=c_1X_1+d_1Y_1$. Due to $\mu_F,\mu_G<\infty$, we have $\E[Z_1]<\infty$. It follows from Lemma \ref{lem:3} that
 \begin{align*}
      \p\(Z_1\in(x,x+1]\)
     = \(F\(\frac{x}{c_1},\frac{x+1}{c_1}\]+G\(\frac{x}{d_1},\frac{x+1}{d_1}\]\)(1+o(1)).
 \end{align*}
 {Based on $F(x,x+1]\asymp G(x,x+1]$, we have
$$\p\(Z_1\in(x,x+1]\)\asymp G\(\frac{x}{c_1},\frac{x+1}{c_1}\]+G\(\frac{x}{d_1},\frac{x+1}{d_1}\].$$
It is obvious that the relation \eqref{eq:lem4.1} holds for $n=m=1$. Now we shall verify the correctness of the relation \eqref{eq:lem4.1} with $n=1$ and $m=2$. Due to the fact that $Z_1$ and $Y_2$ are independent and the same lines as the proof of Lemma \ref{lem:1} with some straightforward 
modifications,  it holds uniformly for $(c_1,d_1,d_2)\in[a,b]^3$ that
\begin{align*}
    \p\(S_1^c+T_2^d>x\)
    &=\p(Z_1>x)+\p(d_2Y_2>x)+d_2\mu_{G}\p(Z_1\in(x,x+1])\nn\\
&\quad +\E[Z_1]\p(d_2Y_2\in(x,x+1])+o\big(\p(Z_1\in(x,x+1])+\p(d_2Y_2\in(x,x+1])\big).
\end{align*}}
Next, it follows from Lemma \ref{lem:3} that the relation \eqref{eq:lem4.1} holds for $n=1$ and $m=2$. Finally, for general situations, without loss of generality, we can use the same method used in the proof of Lemma \ref{lem:2} as well as the recursive method to obtain the desired result.
\end{proof}

\begin{proof}[Proofs of Theorems \ref{the:1} and \ref{the:2}] Noticing that $\{(X_i, Y_i): i\in \N\}$ and $\{\xi_i,\eta_i: i\in \N\}$ are mutually independent and conditioning on $\xi_i=c_i,\;\eta_i=d_i,\;i\in \N$, {by Lemmas \ref{lem:2} and \ref{lem:4},} we can directly obtain Theorems \ref{the:1} and \ref{the:2}.
\end{proof}

\begin{proof}[Proofs of Corollaries \ref{cor:1} and \ref{cor:2}]
According to Lemma 4.7 of \cite{lin2020second}, it follows that $F(x,x+t]\sim tF(x,x+1]$ holds locally uniformly for $t \in (0, \infty)$ (i.e., it holds uniformly for $t \in [0, T]$ for any $T \in (0, \infty)$; see, e.g., \cite{li2010subexponential}). In view of Theorem 3 in \cite{kluppelberg1989estimation}, we have $F(x,x+1]= f(x)(1+o(1))$. Therefore, if $f\in\RV_{-(\alpha+1)}$, then, for all $i=1,\dots,n,$
\begin{align}\label{eq:3.7}
    \p(\xi_i X_i\in(x,x+1])
    &=\int_a^bF\(\frac{x}{t},\frac{x}{t}+\frac{1}{t}\]\p(\xi_i\in\d t)\nn\\
     &=\int_a^b \frac{1}{t}f\(\frac{x}{t}\)\p(\xi_i\in\d t)(1+o(1))\nn\\
     &=\E[\xi_i^\alpha]f(x)(1+o(1)),
\end{align}
Similarly, according to $g\in\RV_{-(\beta+1)}$, it follows that $\p(\eta_j Y_j\in(y,y+1])=\E[\eta_j^\beta]g(y)(1+o(1))$ 
for all $j=1,\dots,m$. 
We introduce four independent rvs $X',X^*,Y'$ and $Y^*$ with corresponding distribution functions $F,F^2,G$ and $G^2$. 
Furthermore, for all $i\neq j$, by \eqref{eq:3.7}, we have
\begin{align}\label{eq:3.9}
     &\quad \E\[\eta_iY_i\I_{\{\xi_iX_i>x,\eta_jY_j\in(y,y+1]\}}\]\nn\\
    &=(1+r)\E\[\eta_iY'\I_{\{\xi_iX'>x,\eta_jY_j\in(y,y+1]\}}\]-r\E\[\eta_iY^*\I_{\{\xi_iX'>x,\eta_jY_j\in(y,y+1]\}}\]\nn\\
    &\quad-r\E\[\eta_iY'\I_{\{\xi_iX^*>x,\eta_jY_j\in(y,y+1]\}}\]+r\E\[\eta_iY^*\I_{\{\xi_iX^*>x,\eta_jY_j\in(y,y+1]\}}\]\nn\\
    &=\(\mu_G+r(\mu_{G^2}-\mu_G)\)\E\[\eta_i\eta_j^\beta\I_{\{\xi_iX_i>x\}}\]g(y)(1+o(1)),
\end{align}
and 
\begin{align}\label{eq:3.10}
     &\quad\E\[\eta_jY_j\I_{\{\xi_iX_i>x,\eta_iY_i\in(y,y+1]\}}\]\nn\\
     &=(1+r)\E\[\eta_jY_j\I_{\{\xi_iX'>x,\eta_iY'\in(y,y+1]\}}\]-r\E\[\eta_jY_j\I_{\{\xi_iX^*>x,\eta_iY'\in(y,y+1]\}}\]\nn\\
     &\quad-r\E\[\eta_jY_j\I_{\{\xi_iX'>x,\eta_iY^*\in(y,y+1]\}}\]+r\E\[\eta_jY_j\I_{\{\xi_iX^*>x,\eta_iY^*\in(y,y+1]\}}\]\nn\\
     &=(1+r)\mu_G\E\[\eta_j\eta_i^\beta\I_{\{\xi_iX_i>x\}}\]g(y)(1+o(1)).
\end{align}
Therefore, by \eqref{eq:3.9} and \eqref{eq:3.10}, it follows that
\begin{align*}
&\quad\sum_{i=1}^n\sum_{j=1}^m\sum_{l\in\Lambda_m^j}\E\[\eta_lY_l\I_{\{\xi_iX_i>x,\eta_jY_j\in(y,y+1]\}}\]\nn\\
&=\sum_{i=1}^n\(\sum_{j\in\Lambda_m^{i}}\(\sum_{l\in\Lambda_m^{j,i}}+\sum_{l=i}^{m}\)+\sum_{j=i}^m\sum_{l\in\Lambda_m^{i}}\)\E\[\eta_lY_l\I_{\{\xi_iX_i>x,\eta_jY_j\in(y,y+1]\}}\]\nn\\
&=\(\sum_{i=1}^n\sum_{j\in\Lambda_m^{i}}\sum_{l\in\Lambda_m^{j,i}}\mu_G\E\[\eta_l\eta_j^\beta\I_{\{\xi_iX_i>x\}}\]+\sum_{i=1}^{n\wedge m}\sum_{j\in\Lambda_m^{i}}\(\mu_G+r(\mu_{G^2}-\mu_G)\)\E\[\eta_i\eta_j^\beta\I_{\{\xi_iX_i>x\}}\]\right.\nn\\
&\quad\left.+\sum_{i=1}^{n\wedge m}\sum_{j\in\Lambda_m^{i}}(1+r)\mu_G\E\[\eta_j\eta_i^\beta\I_{\{\xi_iX_i>x\}}\]\)g(y)(1+o(1))\nn\\
&=\Bigg(\sum_{i=1}^n\sum_{j=1}^m\sum_{l\in\Lambda_m^j}\mu_G\E\[\eta_l\eta_j^\beta\I_{\{\xi_iX_i>x\}}\]+r\sum_{i=1}^{n\wedge m}\sum_{j\in\Lambda_m^{i}}\(\mu_G\E\[\eta_j\eta_i^\beta\I_{\{\xi_iX_i>x\}}\]\right.\nn\\
&\quad\left.+(\mu_{G^2}-\mu_G)\E\[\eta_i\eta_j^\beta\I_{\{\xi_iX_i>x\}}\]\)\Bigg) g(y)(1+o(1)).
\end{align*}
 By the same arguments as above, we can also conclude that
\begin{align*}
&\quad\sum_{i=1}^n\sum_{j=1}^m\sum_{l\in\Lambda_n^i}\E\[\xi_lX_l\I_{\{\xi_iX_i\in(x,x+1],\eta_jY_j>y\}}\]\nn\\
    &=\Bigg(\sum_{i=1}^n\sum_{j=1}^m\sum_{l\in\Lambda_n^i}\mu_F\E\[\xi_l\xi_i^\alpha\I_{\{\eta_jY_j>y\}}\]+r\sum_{j=1}^{n\wedge m}\sum_{i\in\Lambda_n^j}\bigg(\mu_F\E\[\xi_i\xi_j^\alpha\I_{\{\eta_jY_j>y\}}\]\nn\\
    &\quad\quad +(\mu_{F^2}-\mu_F)\E\[\xi_j\xi_i^\alpha\I_{\{\eta_jY_j>y\}}\] \bigg)\Bigg) f(x)(1+o(1)).  
\end{align*}
In addition, by Karamata’s theorem, we have $\overline{F}(x)=\frac{1}{\alpha}xf(x)(1+o(1))$ and 
\begin{align*}
     &\quad\p\(\xi_i X_i>x,\eta_jY_j\in(y,y+1]\)\nn\\
     &=\int_{a_1}^{b_1}\int_{a_2}^{b_2}\p\(X_i>\frac{x}{t},Y_j\in\(\frac{y}{s},\frac{y+1}{s}\]\)\p\(\xi_i\in \d t,\eta_j\in \d s\)\nn\\
     &\asymp\int_{a_1}^{b_1}\int_{a_2}^{b_2}\overline{F}\(\frac{x}{t}\)\overline{G}\(\frac{y}{s},\frac{y+1}{s}\]\p\(\xi_i\in \d t,\eta_j\in \d s\)\nn\\
     &=\int_{a_1}^{b_1}\int_{a_2}^{b_2}\frac{x}{\alpha t}f\(\frac{x}{t}\)g\(\frac{y}{s}\)(1+o(1))\p\(\xi_i\in \d t,\eta_j\in \d s\)\nn\\
    &=\frac{1}{\alpha}xf(x)g(y)\int_{a_1}^{b_1}\int_{a_2}^{b_2} t^{\alpha}s^{\beta+1} \p\(\xi_i\in \d t,\eta_j\in \d s\)(1+o(1))\nn\\
     &=\frac{1}{\alpha}\E\[\xi_i^\alpha\eta_j^{\beta+1}\]xf(x)g(y)(1+o(1))\nn\\
     &\asymp xf(x)g(y)(1+o(1)).
\end{align*}
Similarly, by the simple fact that $x+y\asymp \sqrt{x^2+y^2}$ as $(x,y)\rightarrow(\infty,\infty)$, we have 
\begin{align*}
 xf(x)g(y)+yf(x)g(y)=(x+y)f(x)g(y)\asymp\sqrt{x^2+y^2}f(x)g(y).
\end{align*}
Thus, by Theorem \ref{the:1}, we can directly complete the proof of Corollary \ref{cor:1}. Furthermore, according to Theorem \ref{the:2} and $f,g\in \RV_{-(\alpha+1)}$, we can easily obtain Corollary \ref{cor:2}.
\end{proof}

\setcounter{equation}{0}\par
\section{Application and simulation}\label{sec:3}
 In this section, we first focus on a bidimensional discrete-time risk model and obtain two kinds of tail probabilities. Then we conduct some simulations to check the accuracy of the approximation for the joint tail probability and the sum tail probability of the randomly weighted sum.

\subsection{An application in a  bidimensional  discrete-time risk model}\label{sec:application}

 According to Theorems \ref{the:1} and \ref{the:2}, we can easily establish the second-order expansion for the joint tail probability and the sum tail probability of the stochastic discounted value of aggregate net losses. 


\begin{proposition}\label{the:3.1}
    Consider the bidimensional discrete-time risk model introduced in \eqref{eq:U}, in which  $\{(X,Y),(X_i,Y_i):i\in\N\}$ follow a common iid bivariate FGM distribution \eqref{eq:FGM} with $r\in[-1,1]$. Assume that $F$ is the marginal distribution of $\{X_i: i\in \N\}$ and
 $G$ is the marginal distribution of $\{Y_i: i\in \N\}$. Assume that $\{(X,Y),(X_i,Y_i): i\in \N\}$ and $\{(R_i,\widetilde{R}_i): i\in \N\}$ are mutually independent.  If ${F,G} \in {\widetilde{\mathscr S}_2}$, then for fixed $n\in\N$, it holds that
\begin{align*}
     D_{\text{and}}(x,y;n)
    &=\sum_{i=1}^n\sum_{j=1}^n\p\(X_i\prod\limits_{k=1}^iR_k>x,Y_j\prod\limits_{k=1}^j\widetilde{R}_k>y\)\nn\\
    &\quad+\sum_{i=1}^n\sum_{j=1}^n\(\sum_{l\in\Lambda_n^j}\E\[Y_l\prod\limits_{k=1}^l\widetilde{R}_k\I_{\left\{X_i\prod\limits_{k=1}^iR_k>x,Y_j\prod\limits_{k=1}^j\widetilde{R}_k\in(y,y+1]\right\}}\]\right.\nn\\
&\quad\quad\left.+\sum_{l\in\Lambda_n^i}\E\[X_l\prod\limits_{k=1}^lR_k\I_{\left\{X_i\prod\limits_{k=1}^iR_k\in(x,x+1],Y_j\prod\limits_{k=1}^j\widetilde{R}_k>y\right\}}\]\)\nn\\
&\quad+ {\sum_{i=1}^n\sum_{j=1}^no\(\(\p^2\(X_i\prod\limits_{k=1}^iR_k\in(x,x+1],Y_j\prod\limits_{k=1}^j\widetilde{R}_k>y\)\right.\right.}\nn\\
&\quad\quad {\left.\left.+\p^2\(X_i\prod\limits_{k=1}^iR_k>x,Y_j\prod\limits_{k=1}^j\widetilde{R}_k\in(y,y+1]\)\)^{1/2}\)}.
\end{align*}
\end{proposition}

 If we reduce Proposition \ref{the:3.1} to a unidimensional case, one can easily find that the result obtained in this paper is in accordance with the one in \cite{yang2022seconda}, although they considered a unidimensional discrete-time risk model and assumed that the insurance risk and the financial risk follow an FGM distribution.

Now we establish the second-order asymptotics for the sum tail probability.

\begin{proposition}\label{the:3.2}
    Under the conditions of Proposition \ref{the:3.1}, if we further assume that $F(x,x+1]\asymp G(x,x+1]$ and set $z=x+y$, then
     \begin{align*}
   &\quad D_{\text{sum}}(x,y;n)\nn\\
   &=\sum_{i=1}^n\p\(X_i\prod\limits_{k=1}^iR_k>z\)
 +\mu_F\sum_{1\leq i\neq l\leq n }\E\[\prod\limits_{k=1}^iR_k\I_{\left\{X_l\prod\limits_{k=1}^lR_k
    \in(z,z+1]\right\}}\]\nn\\
    &\quad +\sum_{j=1}^n\p\(Y_j\prod\limits_{k=1}^j\widetilde{R}_k>z\)+\mu_G\sum_{1\leq j\neq l\leq n }\E\[\prod\limits_{k=1}^j\widetilde{R}_k\I_{\left\{Y_l\prod\limits_{k=1}^l\widetilde{R}_k
    \in(z,z+1]\right\}}\]\nn\\
        &\quad+\sum_{i=1}^n\sum_{j=1}^n\(\E\[X_i\prod\limits_{k=1}^iR_k\I_{\left\{Y_j\prod\limits_{k=1}^j\widetilde{R}_k
    \in(z,z+1]\right\}}\]+\E\[Y_j\prod\limits_{k=1}^j\widetilde{R}_k\I_{\left\{X_i\prod\limits_{k=1}^iR_k
    \in(z,z+1]\right\}}\]\)\nn\\
    &\quad+o\(\sum_{i=1}^n\p\(X_i\prod_{k=1}^iR_k
    \in(z,z+1]\)+\sum_{j=1}^n\p\(Y_j\prod_{k=1}^j\widetilde{R}_k
    \in(z,z+1]\)\).
\end{align*}
 \end{proposition}

\subsection{ Simulation study}\label{sec:simulation}
 In this subsection, our aim is to check the precision of the approximation for the joint tail probability and the sum tail probability of the randomly weighted sums in Corollaries \ref{the:1} and \ref{the:2} with any fixed $m,n\in \N$ via the crude Monte Carlo method.  We will illustrate the improvement of our second-order results compared to the first-order asymptotics for $\p(S_m^\xi>x,T_n^\eta>y)$ and $\p(S_m^\xi+T_n^\eta>x+y)$.

Firstly, the generic random pair $(X,Y)$ is linked through a bivariate FGM distribution of the form \eqref{eq:FGM} with $r\in[-1,1]$. Suppose that $n=m=2$. Let the primary random variables $X_1,X_2$ and $Y_1,Y_2$, respectively, follow the two Pareto distributions
$$F(x)=1-\(\frac{k_1}{x+k_1}\)^{\alpha_1}~~~~\mbox{and}~~~~G(x)=1-\(\frac{k_2}{x+k_2}\)^{\alpha_2}~~\mbox{with}~~ \alpha_1,\alpha_2>1,$$
{with density functions $f(x)$ and $g(x)$}.
For random weights $\xi_1,\xi_2$ and $\eta_1,\eta_2$, let $\{\xi,\xi_i:i=1,2\}$ and $\{\eta,\eta_i:i=1,2\}$ be two sequences of uniform i.i.d. rvs on $[a_1,b_1]$ and $[a_2,b_2]$, respectively. 

The calculation procedure for the simulated value of the joint tail probability $\p(S_2^\xi>x,T_2^\eta>y)$ is listed here. We first generate two independent copies of $(X,Y)$ with a simple size $N$ denoted by $(X_i^{(k)},Y_i^{(k)})$, $i=1,2$, and $k=1,\ldots,N$. Then, for each sample $k=1,\ldots,N$, we define
 $$   S_2^{\xi,(k)}=\sum_{i=1}^2\xi_i^{(k)}X_i^{(k)},~~~~\mbox{and}~~~~ 
T_2^{\eta,(k)}=\sum_{j=1}^2\eta_j^{(k)}Y_j^{(k)}.
$$
In this way, the probability of the joint tail $\p(S_2^\xi>x,T_2^\eta>y)$ can be estimated by
\begin{align*}
    \text{Sim}(x,y)=\frac{1}{N}\sum_{k=1}^N\I_{\{S_2^{\xi,(k)}>x,T_2^{\eta,(k)}>y\}}.
\end{align*}
For the asymptotic value of $\p(S_2^\xi>x,T_2^\eta>y)$, we consider the first- and second-order ones, denoted by $\text{Asy}^{(1)}(x,y)$ and $\text{Asy}^{(2)}(x,y)$, respectively. By Corollary \ref{cor:1}, the first-order asymptotic value $\text{Asy}^{(1)}(x,y)$, equal to the first term on the right-hand side
of \eqref{eq:cor1}, is given by
\begin{align*}
    \text{Asy}^{(1)}(x,y)=\sum_{i=1}^2\sum_{j=1}^2\p(\xi_iX_i>x,\eta_jY_j>y),
\end{align*}
and the second-order asymptotic value Asy$^{(2)}(x,y)$, equal to the sum of the first two terms on the right-hand side of \eqref{eq:cor1}, is given by
\begin{align*}
    \text{Asy}^{(2)}(x,y)&= \text{Asy}^{(1)}(x,y)+\(4\mu_G+2r\mu_{G^2}\)\E\[\eta^{\beta+1}\]\p(\xi X>x)g(y)\nn\\
    &\quad+\(4\mu_F+2r\mu_{F^2}\)\E\[\xi^{\alpha+1}\]\p(\eta Y>y)f(x).
\end{align*}

\begin{table}[t]%
  \centering
  \caption{Accuracy of Corollary \ref{the:1} in the case of Pareto distributions with $\alpha_1=2.01$, $k_1=2$, $\alpha_2=2.2$, $k_2=4$, $a_1=a_2=1$, $b_1=b_2=2$ and  $r=0.5$.}
  \label{tab:1}
  \setlength{\tabcolsep}{3.0mm}{
  \begin{tabular}{cccc cc}%
  \hline\hline\noalign{\smallskip}
  $(x,y)$ & Sim$(x,y)$& Asy$^{(1)}(x,y)$& Asy$^{(2)}(x,y)$& $\frac{\text{Asy}^{(1)}(x,y)}{\text{Sim}(x,y)}$ &$\frac{\text{Asy}^{(2)}(x,y)}{\text{Sim}(x,y)}$ \\
\noalign{\smallskip}\hline\noalign{\smallskip}
 $(20,25)$  &   $3.12\times10^{-4}$ &  $2.24\times10^{-4}$& $2.91\times10^{-4}$ &  $0.7179$& $0.9327$\\
 $(25,30)$&   $1.45\times10^{-4}$&   $1.07\times10^{-4}$ & $1.36\times10^{-4}$& $0.7379$& $0.9379$\\
 $(30,35)$&   $7.42\times10^{-5}$&   $5.66\times10^{-5}$& $6.98\times10^{-5}$& $0.7628$ &$0.9407$\\
 $(35,40)$& $4.22\times10^{-5}$&   $3.32\times10^{-5}$& $4.00\times10^{-5}$& $0.7867$ &$0.9479$\\
  $(40,45)$& $2.60\times10^{-5}$ & $2.09\times10^{-5}$& $2.49\times10^{-5}$  &$0.8038$& $0.9578$\\
  $(45,50)$& $1.60\times10^{-5}$ & $1.30\times10^{-5}$& $1.52\times10^{-5}$  &$0.8125$& $0.9500$\\ 
$(50,55)$&$1.05\times10^{-5}$ & $8.63\times10^{-6}$ & $1.01\times10^{-5}$& $0.8219$& $0.9619$\\
$(55,60)$&$7.35\times10^{-6}$& $6.42\times10^{-6}$& $7.34\times10^{-6}$& $0.8735$& $0.9986$\\
  \noalign{\smallskip}\hline
  \end{tabular}}
\end{table}

Set the sample size $N = 10^7$ and repeat 10 times. Table \ref{tab:1} summarizes the simulated value $\text{Sim}(x,y)$, the first-order asymptotic value $\text{Asy}^{(1)}(x,y)$ and the second-order asymptotic value $\text{Asy}^{(2)}(x,y)$ of the joint tail probability $\p(S_2^\xi>x,T_2^\eta>y)$, as well as the corresponding ratios with respect to different $x$ and $y$. It can be seen from the table that the simulated values and the two asymptotic values are closer and gradually decrease with respect to $x$ and $y$. This indicates that, compared with the first-order asymptotics, the second-order asymptotic values are more precise, which is also shown by the ratios between the asymptotic and simulated values.

Similarly, the sum probability of the tail $\p(S_2^\xi+T_2^\eta>x)$ can be estimated by
\begin{align*}
    \text{Sim}(x)=\frac{1}{N}\sum_{k=1}^N\I_{\{S_2^{\xi,(k)}+T_2^{\eta,(k)}>x\}}.
\end{align*}
For the asymptotic value of $\p(S_2^\xi+T_2^\eta>x)$, we consider the first- and second-order ones, denoted by $\text{Asy}^{(1)}(x)$ and $\text{Asy}^{(2)}(x)$, respectively. By Corollary \ref{cor:2}, the first-order asymptotic value $\text{Asy}^{(1)}(x)$, equal to the first term on the right-hand side
of \eqref{eq:cor2}, is given by
\begin{align*}
    \text{Asy}^{(1)}(x)=\sum_{i=1}^2\sum_{j=1}^2\p(\xi_iX_i+\eta_jY_j>x),
\end{align*}
and the second-order asymptotic value Asy$^{(2)}(x)$, equal to the sum of the first two terms on the right-hand side of \eqref{eq:cor2}, is given by
\begin{align*}
    \text{Asy}^{(2)}(x)&=\text{Asy}^{(1)}(x)+2\mu_F\E\[\xi^{\alpha+1}\]f(x)+2\mu_G\E\[\eta^{\alpha+1}\]g(x) \nn\\
&\quad+\(4\mu_F-2r(\mu_{F^2}-\mu_F)\)\E\[\xi\eta^\alpha\]g(x)+\(4\mu_G-2r(\mu_{G^2}-\mu_G)\)\E\[\xi^\alpha\eta\]f(x).
\end{align*}

\begin{table}[t]%
  \centering
  \caption{Accuracy of Corollary \ref{the:2} in the case of Pareto distributions with $\alpha_1=\alpha_2=2.01$, $k_1=k_2=1$, $a_1=a_2=1$, $b_1=b_2=2$ and $r=0.6$.} 
  \label{tab:2}
  \setlength{\tabcolsep}{4.0mm}{
  \begin{tabular}{cccc cc}%
  \hline\hline\noalign{\smallskip}
  $x$ & Sim$(x)$& Asy$^{(1)}(x)$& Asy$^{(2)}(x)$& $\frac{\text{Asy}^{(1)}(x)}{\text{Sim}(x)}$ &$\frac{\text{Asy}^{(2)}(x)}{\text{Sim}(x)}$ \\
\noalign{\smallskip}\hline\noalign{\smallskip}
 $10$  &   $3.89\times10^{-2}$ &  $2.62\times10^{-2}$& $3.65\times10^{-2}$ &  $0.6735$& $0.9383$\\
 $20$&   $9.30\times10^{-3}$&   $7.31\times10^{-3}$ & $8.91\times10^{-3}$& $0.7860$& $0.9581$\\
 $30$&   $4.02\times10^{-3}$&   $3.41\times10^{-3}$& $3.93\times10^{-3}$& $0.8483$ &$0.9776$\\
  $40$& $2.20\times10^{-3}$ & $1.94\times10^{-3}$& $2.16\times10^{-3}$  &$0.8818$& $0.9818$\\
  $50$& $1.40\times10^{-3}$ & $1.27\times10^{-3}$& $1.39\times10^{-3}$  &$0.9071$& $0.9929$\\ 
$60$&$9.48\times10^{-4}$ & $8.70\times10^{-4}$ & $9.42\times10^{-4}$& $0.9177$& $0.9937$\\
$70$&$6.82\times10^{-4}$ & $6.36\times10^{-4}$ & $6.83\times10^{-4}$& $0.9323$& $1.001$\\
$80$&$5.40\times10^{-4}$ & $5.10\times10^{-4}$ & $5.40\times10^{-4}$& $0.9444$& $1.000$\\
  \noalign{\smallskip}\hline
  \end{tabular}}
\end{table}
From Table \ref{tab:2}, we again verify that the second-order asymptotics are much more accurate than the first-order ones for the sum tail probability.


\section{Conclusion}\label{sec:conc}
This paper first establishes the second-order asymptotic formulas for the joint and sum tail probabilities of randomly weighted sums under the assumption that the primary rvs satisfy a commonly bivariate FGM distribution with second-order subexponential marginal tails and the random weights are arbitrarily dependent, but independent of the primary rvs. Based on these results,  some asymptotic estimations for the joint and sum tail probabilities for stochastic discounted aggregate net losses of a bidimensional discrete-time risk model are further derived. 			

In addition, as an initial attempt to explore the second-order asymptotic theory of randomly weighted sums with a second-order subexponential tail, our results can be extended to various models for second-order asymptotics: continuous-time bidimensional risk models, multidimensional continuous-time or discrete-time risk models and a wide class of risk measure models, including, e.g., the Joint Expected Shortfall (JES) (\cite{ji2021tail}) and the Joint Marginal Expected Shortfall (JMES) (\cite{pu2024joint}). The generalized risk measures (GRM) framework proposed in \cite{fadina2024framework} is also of interest for further investigation. 
 
\vspace{6mm}
\noindent{\normalsize \bf Acknowledgements} \vskip0.1in\parskip=0mm

\noindent  
The authors would like to thank the editors, two anonymous referees and Zhichen Wang for their valuable comments and insightful suggestions that help us greatly improve the paper. Geng's work is supported by the Provincial Natural Science Research Project of Anhui Colleges (2024AH050037) and the doctoral research initiation fund of Anhui University (s020318033/015). Liu acknowledges financial support from the National Natural Science Foundation of China (Grant No. 12401624), Guangdong Science and Technology Program (Grant No. 2024QN11X076), Shenzhen Science and Technology Program (Grant Nos. RCBS20231211090814028, JCYJ2025060 4141203005, 2025TC0010) and The Chinese University of Hong Kong (Shenzhen) University Development Fund (Grant No. UDF01003336) and is partly supported by the Guangdong Provincial Key Laboratory of Mathematical Foundations for Artificial Intelligence (Grant No. 2023B1212010001). Wang's work is supported by the Provincial Natural Science Research Project of Anhui Colleges (2022AH050067).

\small
\bibliographystyle{apalike}
\bibliography{reference}

\end{document}